\documentclass[12pt]{amsart}
\usepackage{amscd,amsmath,amsthm,amssymb,graphics}
\usepackage{amsfonts,amssymb,amscd,amsmath,enumitem,verbatim}

\theoremstyle{plain}
\newtheorem{Theorem}{Theorem}
\newtheorem{Lemma}[Theorem]{Lemma}
\newtheorem{Corollary}[Theorem]{Corollary}
\newtheorem{Proposition}[Theorem]{Proposition}

\theoremstyle{definition}
\newtheorem{Definition}[Theorem]{Definition}
\newtheorem{Remark}[Theorem]{Remark}

\newtheorem{Example}[Theorem]{Example}

\usepackage{float}
\begin{document}
\title{Convergence conditions for $p$--adic continued fractions}
\author {Nadir Murru, Giuliano Romeo,  Giordano Santilli}

\begin{abstract}
Continued fractions have been introduced in the field of $p$--adic numbers $\mathbb Q_p$ by several authors. However, a standard definition is still missing since all the proposed algorithms are not able to replicate all the properties of continued fractions in $\mathbb R$. In particular, an analogue of the Lagrange's Theorem is not yet proved for any attempt of generalizing continued fractions in $\mathbb Q_p$.  Thus, it is worth to study the definition of new algorithms for $p$–adic continued fractions.
The main condition that a new method needs to fulfill is the convergence in $\mathbb Q_p$ of the continued fractions.
In this paper we study some convergence conditions for continued fractions in $\mathbb{Q}_p$. These results allow to define many new families of continued fractions whose convergence is guaranteed. Then we provide some new algorithms exploiting the new convergence condition and we prove that one of them terminates in a finite number of steps when the input is rational, as it happens for real continued fractions.

\end{abstract}

\maketitle
\section{Introduction}
In 1940, Mahler \cite{Mah} gave the first idea for introducing continued fractions in the field of $p$--adic numbers $\mathbb Q_p$. Starting from this, several authors studied the problem of defining an algorithm for expanding elements of $\mathbb Q_p$ in continued fractions. The most notable results were provided by Browkin \cite{BI}, Ruban \cite{RUB} and Schneider \cite{SCH} who defined different $p$--adic continued fractions algorithms with the aim of obtaining the same good properties that hold in the real case. However, all these algorithms fail in the attempt of characterizing quadratic irrationals by periodic continued fractions, as in the case of $\mathbb R$. The study of the periodicity of these algorithms have been deepened by several authors. Schneider's algorithm is not periodic for all quadratic irrationals, but there is an effective criterion to forecast when this happens (see \cite{VP, TIL, DEWII}).
Ooto \cite{OO} proved that an analogue of Lagrange's Theorem does not hold for Ruban's continued fractions and Capuano et al. \cite{CVZ} gave an effective condition to check the periodicity. Moreover, Ruban and Schneider algorithms provide finite or periodic expansion for rationals.
Browkin's algorithm is of particular interest since it always gives finite representations for rational numbers, but it is not known if an analogue of the Lagrange's Theorem holds. In \cite{BEI, BEII}, the authors proved some results about the periodicity of this algorithm and Capuano et al. \cite{CMT} gave some necessary and sufficient conditions for periodicity, but such conditions do not allow to prove that an analogue of Lagrange's Theorem does not hold. From experimental results, it seems very unlikely that Browkin's algorithm provides periodic expansion for any quadratic irrational. For this reason, in 2000, Browkin himself defined a new algorithm \cite{BII} and it has been proved in \cite{BCMI} that also this second algorithm produces a finite continued fraction for rational numbers. Browkin's second algorithm works better on quadratic irrationals, but also in this case they do not always present periodic expansions in continued fractions. The periodicity of this algorithm has been deepened in \cite{MRS}. Further studies on $p$--adic continued fractions can be found in \cite{DEA, LAO, WANI, WANII}. Thus, it is worth to study the definition of new algorithms for $p$--adic continued fractions. It is believed that some slight modification of Browkin's second algorithm \cite{BII} can give a periodic continued fraction for all quadratic irrationals in $\mathbb{Q}_p$, without losing the finite representation for the rationals.
With this purpose in mind, the first condition that a new method needs to fulfill is the convergence in $\mathbb Q_p$ of the continued fractions produced by the algorithm.


In this paper, we give a sufficient condition on the partial quotients of a $p$--adic continued fractions in order to achieve the convergence in $\mathbb{Q}_p$. In particular, we study a condition that allows to extend the idea of Browkin in \cite{BII}, giving space to several possible new definitions of $p$--adic continued fractions. Exploiting this condition, we then propose a new $p$--adic continued fraction algorithm that is a natural generalization of the construction performed in \cite{BII} for the second algorithm of Browkin. Moreover, 
we also prove that this new algorithm terminates in a finite number of steps on each $\alpha\in\mathbb{Q}$.

 \vspace{-0.07cm}
\section{Preliminaries}

Let us denote with $v_p(\cdot)$ and $|\cdot|_p$, respectively, the $p$--adic valuation and the $p$--adic absolute value over $\mathbb{Q}$, where $p$ is an odd prime. The Euclidean norm will be denoted as usual by $|\cdot|$.
We denote a continued fraction of a value $\alpha$ with the usual notation as  
\[\alpha = b_0 + \cfrac{1}{b_1 + \cfrac{1}{b_2 + \cfrac{1}{\ddots}}} = [b_0, b_1, b_2, \ldots].\] 
Moreover, we call $\frac{A_n}{B_n}$, for all $n\in\mathbb{N}$, the convergents of the continued fraction, that may be defined recursively by using the well-known formulas 
\[
\begin{cases}
A_0=b_0,\\
A_1=b_1b_0+1,\\
A_n=b_nA_{n-1}+A_{n-2} \text{ for } n \geq 2,
\end{cases}
\begin{cases}
B_0=1,\\
B_1=b_1,\\
B_n=b_nB_{n-1}+B_{n-2} \text{ for } n \geq 2.
\end{cases}
\]
The first important requirement when designing an algorithm for $p$--adic continued fractions is that all the expansions converge to a $p$--adic number, that is
\[\lim\limits_{n\rightarrow +\infty} \frac{A_n}{B_n}=\alpha\in\mathbb{Q}_p.\]
The first algorithm proposed by Browkin in \cite{BI} works as follows.
Starting from an input $\alpha_0\in\mathbb{Q}_p$ then the partial quotients of the $p$--adic continued fraction are evaluated by
\begin{equation}\label{Br1}
\begin{cases}
b_n=s(\alpha_n)\\
\alpha_{n+1}=\frac{1}{\alpha_n-b_n},
\end{cases} \quad n \geq0
\end{equation}
where $s:\mathbb{Q}_p\rightarrow \mathbb{Q}$ is defined by
\[s(\alpha)=\sum\limits_{n=-r}^{0} a_n p^n\in\mathbb{Q},\]
for a $p$--adic number $\alpha=\sum\limits_{n=-r}^{+\infty} a_np^n\in\mathbb{Q}_p$, with $r\in\mathbb{Z}$ and $a_n\in \{-\frac{p-1}{2},\ldots,\frac{p-1}{2}\}$. In this algorithm, the function $s$ plays the same role of the floor function in the classical algorithm of continued fractions in $\mathbb R$. Ruban's algorithm \cite{RUB} employs the same function $s$, with the only difference that the representatives are taken in $\{0,\ldots, p-1 \}$. More than 20 years later, Browkin defines another algorithm in \cite{BII}, where starting from $\alpha_0\in\mathbb{Q}_p$, the partial quotients $b_n$, for $n \geq 0$, are evaluated by
\begin{align} 
\begin{cases}\label{Br2}
b_n=s(\alpha_n) \ \ \ \ \ & \textup{if} \ n \ \textup{even}\\
b_n=t(\alpha_n) & \textup{if} \ n \ \textup{odd}\ \textup{and} \ v_p(\alpha_n-t(\alpha_n))= 0\\
b_n=t(\alpha_n)-sign(t(\alpha_n)) & \textup{if} \ n \ \textup{odd} \ \textup{and} \ v_p(\alpha_n-t(\alpha_n))\neq 0\\
\alpha_{n+1}=\frac{1}{\alpha_n-b_n},
\end{cases}
\end{align}
where $t:\mathbb{Q}_p\rightarrow \mathbb{Q}$ is another function defined for any $p$--adic value $\alpha=\sum\limits_{n=-r}^{+\infty} a_np^n$ as
\[ t(\alpha)=\sum\limits_{n=-r}^{-1}a_np^n, \]
with $r\in\mathbb{Z}$ and $a_n\in \{-\frac{p-1}{2},\ldots,\frac{p-1}{2}\}$. 
In the following we will refer to \eqref{Br1} and \eqref{Br2} respectively as \textit{Browkin I} and \textit{Browkin II}.

The convergence in $\mathbb{Q}_p$ of the continued fractions generated by \textit{Browkin I} is based on the following lemma.

\begin{Lemma}[\cite{BI}, Lemma 1]\label{ConvBr1}
Let an infinite sequence $b_0,b_1,\ldots\in \mathbb{Z}[\frac{1}{p}]$ such that $v_p(b_{n})<0$, for all $n \geq 1$. Then the continued fraction $[b_0,b_1,\ldots]$ is convergent to a $p$--adic number.
\end{Lemma}
In fact, the partial quotients $b_n$ arising from \textit{Browkin I}, for $n\geq 1$, all have negative valuations.\\
 
For what concerns \textit{Browkin II}, the $p$--adic convergence relies on the following lemma.
\begin{Lemma}[\cite{BII}, Lemma 1]\label{ConvBr2}
Let an infinite sequence $b_0,b_1,\ldots\in \mathbb{Z}[\frac{1}{p}]$ such that, for all $n\in\mathbb{N}$,
\begin{equation}
\begin{cases}
v_p(b_{2n})=0\\
v_p(b_{2n+1})<0.
\end{cases}
\end{equation}
Then the continued fraction $[b_0,b_1,\ldots]$ is convergent to a $p$--adic number.
\end{Lemma}

\begin{Remark}\label{rema2}
The proofs of Lemma \ref{ConvBr1} and Lemma \ref{ConvBr2} exploit the strict decrease of the sequence of valuations $v_p(B_n{B_{n+1}})$. 
Moreover, requiring the sequence $v_p(B_{n}B_{n+1})$ strictly decreasing is equivalent to ask that $v_p(B_{n+1})<v_p(B_{n-1})$ for all $n\geq 1$. Thus, the divergence of the sequence of valuations implies the convergence of the correspondent $p$--adic continued fraction. Indeed, in this way we have that
\[ \lim_{n \to \infty} v_p\left( \frac{A_{n+1}}{B_{n+1}} - \frac{A_n}{B_n} \right) = \lim_{n \to \infty} - v_p(B_nB_{n+1}) = +\infty, \]
and 
\[\left|\frac{A_{m}}{B_{m}}-\frac{A_n}{B_n}\right|_p=\left|\frac{A_{n+1}}{B_{n+1}}-\frac{A_n}{B_n}\right|_p=\left|\frac{(-1)^n}{B_{n}B_{n+1}}\right|_p=p^{v_p(B_nB_{n+1})},\]
 proving that $\left\{ \frac{A_n}{B_n} \right\} _{n\in\mathbb{N}}$ is a Cauchy sequence and therefore convergent in $\mathbb Q_p$.
\end{Remark}

\section{Convergence of $p$--adic continued fractions}
The reduction of the number of partial quotients having negative valuations shows better properties in terms of the periods of quadratic irrationals, as pointed out in \cite{BII}. Therefore a promising approach for the definition of a new algorithm should be a further modification of \textit{Browkin II}: we may define a ``$3$-steps''-algorithm that generates the partial quotients such that, for all $n \in\mathbb{N}$, 
\begin{equation}\label{Br3}
\begin{cases}
v_p(b_{3n+1})<0\\
v_p(b_{3n+2})=0\\
v_p(b_{3n+3})=0.
\end{cases}
\end{equation}

Such a construction turns out to be more complex than the previous two algorithms defined by Browkin. In the following example we show that a sequence having these constraints does not converge without a stronger hypothesis. In particular, for every prime $p$, we may construct a suitable continued fraction that does not converge to any $p$--adic number.

\begin{Example}\label{controex}
Let $p$ be an odd prime. We are going to show that there exists a sequence $b_0,b_1,\ldots\in\mathbb{Q}_p$  with, for all $n\in\mathbb{N}$,
\[\begin{cases}
v_p(b_{3n+1})<0\\
v_p(b_{3n+2})=0\\
v_p(b_{3n+3})=0,
\end{cases}\]
such that the sequence $v_p(B_nB_{n+1})$ does not diverge to $-\infty$. Let us define $b_{1}=\frac{1}{p}$. The first denominators of the convergents are
\begin{align*}
B_0&=1,\\
B_1&=b_1=\frac{1}{p},\\
B_2&=b_2B_1+B_0=\frac{b_2+p}{p},\\
B_3&=b_3B_2+B_1=\frac{(b_3b_2+1)+b_3p}{p}.
\end{align*}
Their valuations are
\begin{align*}
v_p(B_0)&=v_p(1)=0,\\
v_p(B_1)&=v_p\Big(\frac{1}{p}\Big)=-1,\\
v_p(B_2)&=v_p\Big(\frac{b_2+p}{p}\Big)=-1,\\
v_p(B_3)&=v_p\Big(\frac{(b_3b_2+1)+b_3p}{p}\Big).
\end{align*}
Let us choose suitable $b_2$ and $b_3$ such that $b_3b_2+1=p$ (for example, $b_2=2$ and $b_3=\frac{p-1}{2}$). Then
\[v_p(B_3)=v_p\Big(\frac{b_3p+p}{p}\Big)=v_p(b_3+1)\geq 0.\] 
At this point,
for a generic $n\in\mathbb{N}$ for which
\[v_p(B_{3n+1})=-1, \ v_p(B_{3n+2})=-1, \ v_p(B_{3n+3})\geq 0,\]
we are going to show that there exists a choice for the partial quotients such that
\[v_p(B_{3(n+1)+1})=-1, \ v_p(B_{3(n+1)+2})=-1, \ v_p(B_{3(n+1)+3})\geq 0.\]
We can write
\begin{align*}
B_{3n+1}&=\frac{a_1}{p}, \ &\textup{with} \ v_p(a_1)&=0,\\
B_{3n+2}&=\frac{a_2}{p}, \ &\textup{with} \ v_p(a_2)&=0,\\
B_{3n+3}&=a_3, \ &\textup{with} \ v_p(a_3)&\geq 0.
\end{align*}
We have two cases: 
\begin{itemize}
    \item 
In the case that $v_p(a_3+a_2)=0$, we choose $b_{3n+4}=\frac{1}{p}$. Therefore,
\[B_{3n+4}=b_{3n+4}B_{3n+3}+B_{3n+2}=\frac{a_3+a_2}{p}.\]
Its valuation is
\[v_p(B_{3n+4})=v_p(a_3+a_2)-v_p(p)=-1,\]
so that we can write $B_{3n+4}=\frac{a_4}{p}$, with $v_p(a_4)=0$. Subsequently,
\[
B_{3n+5}=b_{3n+5}B_{3n+4}+B_{3n+3}=b_{3n+5}\frac{a_4}{p}+a_3=\frac{b_{3n+5}a_4+a_3p}{p},
\]
so that $v_p(B_{3n+5})=-1$. It means that $B_{3n+5}=\frac{a_5}{p}$, with $v_p(a_5)=0$. At the following step,
\[B_{3n+6}=b_{3n+6}B_{3n+5}+B_{3n+4}=\frac{b_{3n+6}a_5+a_4}{p}.\]
Notice that $a_4$ and $a_5$ are arbitrary nonzero elements and we can choose a suitable $b_{3n+6}$ such that
\[b_{3n+6}a_5+a_4\equiv 0 \bmod p.\]
We obtain that $p$ divides $b_{3n+6}a_5+a_4$ and so $v_p(B_{3n+6})\geq 0$.
In this case we have obtained that, starting from
\[v_p(B_{3n+1})=-1, \ v_p(B_{3n+2})=-1, \ v_p(B_{3n+3})\geq 0,\]
then
\[v_p(B_{3(n+1)+1})=-1, \ v_p(B_{3(n+1)+2})=-1, \ v_p(B_{3(n+1)+3})\geq 0.\]
\item
Let us examine also the case $v_p(a_3+a_2)>0$. Here we choose $b_{3n+4}=\frac{2}{p}$. Since $v_p(a_2)=0$ and $v_p(a_3+a_2)>0$, necessarily also $v_p(a_3)=0$. The next denominator is
\[B_{3n+4}=b_{3n+4}B_{3n+3}+B_{3n+2}=\frac{2a_3+a_2}{p}.\]
Notice that since $p$ divides $a_3+a_2$ but does not divide $a_3$, it can not divide $2a_3+a_2$. In this way $v_p(2a_3+a_2)=0$ and
\[v_p(B_{3n+4})=v_p(2a_3+a_2)-v_p(p)=-1.\]
Then we get
\[v_p(B_{3n+5})=v_p(b_{3n+5}B_{3n+4}+B_{3n+3})=-1,\]
and so we can write
\begin{align*}
B_{3n+4}&=\frac{a_4}{p}, \ &\textup{with} \ v_p(a_4)&=0,\\
B_{3n+5}&=\frac{a_5}{p}, \ &\textup{with} \ v_p(a_5)&=0.
\end{align*}
At the next step we have
\[B_{3n+6}=b_{3n+6}B_{3n+5}+B_{3n+4}=\frac{b_{3n+6}a_5+a_4}{p}.\]
As before, we choose $b_{3n+6}$ such
\[b_{3n+6}a_5+a_4\equiv 0 \bmod p.\]
In this way we get $v_p(B_{3n+6})\geq 0$. Hence, also in this second case we have obtained that
\[v_p(B_{3(n+1)+1})=-1, \ v_p(B_{3(n+1)+2})=-1, \ v_p(B_{3(n+1)+3})\geq 0.\]
\end{itemize}
We have just constructed a sequence of denominators $B_n$ such that the sequence of valuations $v_p(B_{n}B_{n+1})=v_p(B_{n})+v_p(B_{n+1})$ can not diverge to $-\infty$. In fact, in particular, $v_p(B_n)\geq -1$ for all $n\in\mathbb{N}$ and the $p$--adic continued fraction is not convergent.
\end{Example}

Starting from the observations of the last example, we would like to characterize the strict decrease of the sequence $v_p(B_nB_{n+1})$ in general. From Remark \ref{rema2}, it is sufficient to investigate the condition $v_p(B_{n+1})<v_p(B_{n-1})$ for all $n\geq 1$.\\

In the following, $b_0,b_1,\ldots$ are elements of $\mathbb{Q}_p$. In fact, as we are going to see in the next results, Browkin's hypotesis of $b_n\in\mathbb{Z}[\frac{1}{p}]$ for all $n\in\mathbb{N}$, seen in Lemma \ref{ConvBr1}  and Lemma \ref{ConvBr2}, is not needed.

\begin{Lemma}\label{lem1}
For all $n\geq 1$, if $v_p(B_{n+1})<v_p(B_{n-1})$, then
\[v_p(B_{n+1})\leq v_p(B_n).\]
\begin{proof}
Let us recall that
\[v_p(B_{n+1})=v_p(b_{n+1}B_n+B_{n-1})\geq \min \{v_p(b_{n+1}B_n),v_p(B_{n-1}) \},\]
with the equality for $v_p(b_{n+1}B_n)\neq v_p(B_{n-1})$.\\
If $v_p(b_{n+1}B_n)< v_p(B_{n-1})$, then
\[v_p(B_{n+1})=v_p(b_{n+1}B_n)=v_p(b_{n+1})+v_p(B_n)\leq v_p(B_n),\]
since $v_p(b_{n+1})\leq 0$. Instead, if $v_p(b_{n+1}B_n)\geq v_p(B_{n-1})$,
\[v_p(B_{n+1})\geq \min \{v_p(b_{n+1}B_n),v_p(B_{n-1})\}= v_p(B_{n-1}),\]
but it is a contradiction with the hypothesis of $v_p(B_{n+1})<v_p(B_{n-1})$, hence this second case can not occur.
\end{proof}
\end{Lemma}

On the other hand it is also possible to prove the following equivalence.

\begin{Lemma}\label{lem2}
For all $n\geq 1$, $v_p(B_{n+1})<v_p(B_{n-1})$ if and only if 
\[v_p(b_{n+1}B_{n})<v_p(B_{n-1}).\]
\begin{proof}
If $v_p(B_{n+1})<v_p(B_{n-1})$ and $v_p(b_{n+1}B_{n})\geq v_p(B_{n-1})$, then
\[v_p(B_{n+1})\geq\min \{v_p(b_{n+1}B_{n}),v_p(B_{n-1})\}=v_p(B_{n-1}),\]
but this contradicts the hypothesis.\\
Conversely, if $v_p(b_{n+1}B_{n})<v_p(B_{n-1})$, then
\[v_p(B_{n+1})=v_p(b_{n+1}B_{n})<v_p(B_{n-1}),\]
and the claim is proved. 
\end{proof}
\end{Lemma}

Using the results obtained above, we may prove the following theorem on the characterization of the strict decrease of the sequence $v_p(B_nB_{n+1})$.

\begin{Theorem} \label{teoconve}
The following conditions are equivalent:
\begin{enumerate}
\item[i)]$v_p(b_{n+1}B_{n})<v_p(B_{n-1})$, for all $n\geq 1$, 
\item[ii)]$v_p(b_nb_{n+1})<0$, for all $n\geq 1$.
\end{enumerate}
\begin{proof}
$i)\Rightarrow ii)$\\
Let us suppose that $v_p(b_{n+1}B_{n})<v_p(B_{n-1})$ for all $n\geq 1$.
\\If $v_p(b_{n+1})<0$, then $v_p(b_{n+1}b_{n})=v_p(b_{n+1})+v_p(b_{n})<0$ and the claim is proved. Therefore, let us assume $v_p(b_{n+1})=0$ and we prove that $v_p(b_{n})<0$.
Since $v_p(b_{n+1})=0$ and
\[v_p(b_{n+1}B_{n})<v_p(B_{n-1}),\]
then $v_p(B_{n})<v_p(B_{n-1})$. The latter means that:
\[v_p(B_{n})=v_p(b_{n}B_{n-1}+B_{n-2})<v_p(B_{n-1}).\]
Moreover, $v_p(B_n)=v_p(b_{n}B_{n-1})$ because otherwise $v_p(B_n)\geq v_p(B_{n-2})$ and this leads to a contradiction, by Lemma \ref{lem2}. Hence, we have obtained that
\[v_p(B_n)=v_p(b_{n}B_{n-1})=v_p(b_{n}) +v_p(B_{n-1})<v_p(B_{n-1}),\]
where the last inequality implies $v_p(b_{n})<0$ and this concludes the proof.\\
$ii)\Rightarrow i)$\\
Conversely, let us suppose that $v_p(b_nb_{n+1})<0$ for all $n\geq 1$. We prove the claim by induction on $n$.\\
\textbf{Base step:}\\
By hypotesis, we have that $v_p(b_1b_2)<0$ and $v_p(b_2b_3)<0$. Hence, for $n=1$ and $n=2$, we have that:
\begin{align*}
v_p(b_2B_1)&=v_p(b_2b_1)<0=v(1)=v(B_0),\\
v_p(b_3B_2)&=v_p(b_3b_2b_1+b_3)=v_p(b_3b_2b_1)=v_p(b_3b_2)+v_p(b_1)<\\
&<v_p(b_1)=v_p(B_1).
\end{align*}
\textbf{Induction step:}\\
Let us suppose that the thesis is true until a step $n\geq 2$ and we show it for $n+1$. 
From $v_p(b_{n+2}b_{n+1})<0$ we get that either $v_p(b_{n+2})<0$ or $v_p(b_{n+1})<0$ (or both).\\ \ \\
\textbf{Case $v_p(b_{n+2})<0$:}\\
In this case, using inductive hypothesis and Lemma \ref{lem1} we get that $v_p(B_{n+1})\leq v_p(B_n)$, hence:
\[v_p(b_{n+2}B_{n+1})=v_p(b_{n+2})+v_p(B_{n+1})<v_p(B_{n+1})\leq v_p(B_n).\]\ \\
\textbf{Case $v_p(b_{n+1})<0$:}\\
In this case we have
\[
b_{n+2}B_{n+1}=b_{n+2}\left(b_{n+1}B_{n} + B_{n-1} \right), 
\]
therefore 
\[
v_p \left(b_{n+2}B_{n+1} \right) \leq v_p\left(b_{n+1}B_{n} + B_{n-1} \right).
\]
The inductive hypothesis ensures that $v_p\left(b_{n+1}B_{n}  \right) < v_p(B_{n-1})$, so 
\[
v_p \left(b_{n+2}B_{n+1} \right) \leq v_p\left(b_{n+1}B_{n} \right) < v_p(B_n)
\]
and this concludes the proof.
\end{proof}
\end{Theorem}

We easily obtain the following corollary, fully characterizing the strict decrease of the sequence of denominators.

\begin{Corollary}
The sequence $\{v_p(B_nB_{n+1})\}_{n\in\mathbb{N}}$ is strictly decreasing if and only if $v_p(b_nb_{n+1})<0$ for all $n\in\mathbb{N}$.
\end{Corollary}

In other words, we have proved that the definition of two consecutive partial quotients with zero valuation makes us lose the strict decrease of the valuation. Moreover, the sufficiency of this condition means that every possible definition in this range works. It would be interesting to study some algorithms that satisfy this hypothesis, different from \textit{Browkin I} and \textit{Browkin II}. For example, it is possible to define $2$ negative partial quotients each $3$ steps or partial quotients that are not in $\mathbb{Z}[\frac{1}{p}]$, as long as the condition, $v_p(b_nb_{n+1})<0$ for all $n\in\mathbb{N}$, is satisfied.

\section{Design of a new algorithm}
In Example \ref{controex} we have showed that an algorithm generating the partial quotients as in $(\ref{Br3})$ never assures the $p$--adic convergence of the continued fraction. Moreover, we have characterized the strict decrease of the sequence $v_p(B_nB_{n+1})$.

However, for the negative divergence of this sequence, we do not need it to be strictly decreasing. So we may wonder in which cases it diverges although it is not strictly decreasing.

What we are going to see here is that adding one additional constraint on the two partial quotients having null valuation it is possible to avoid the growth of the valuation of the denominators $B_n$. In this way we succeed to obtain the convergence of a $p$--adic continued fraction with only one partial quotient with negative valuation each three steps, as defined in $(\ref{Br3})$.

\begin{Theorem}\label{ConvBr3}
Let $b_0,b_1,\ldots \in \mathbb{Q}_p$ such that, for all $n\in\mathbb{N}$:
\[\begin{cases}
v_p(b_{3n+1})<0\\
v_p(b_{3n+2})=0\\
v_p(b_{3n+3})=0.
\end{cases}\]
If $v_p(b_{3n+3}b_{3n+2}+1)=0$ for all $n \in \mathbb{N}$, then, 
\[v_p(B_{3n-2})=v_p(B_{3n-1})=v_p(B_{3n})>v_p(B_{3n+1}).\]
\begin{proof}
Let us prove the claim by induction on $n$.\\
\textbf{Base step}:
\begin{align*}
v_p(B_0)&=v_p(1)=0,\\
v_p(B_{1})&=b_1<0,\\
v_p(B_2)&=v_p(b_2b_1+1)=v_p(b_2)+v_p(b_1)=v_p(b_1)=v_p(B_1),\\
v_p(B_3)&=v_p(b_3B_2+B_1)=v_p((b_3b_2+1)B_1+b_3B_0)\\
&=v_p((b_3b_2+1)B_1)=v_p(B_1)=v_p(B_2),\\
v_p(B_4)&=v_p(b_4B_3+B_2)=v_p(b_4)+v_p(B_3)<v_p(B_3)=\\
&=v_p(B_1)=v_p(B_2),
\end{align*}
where we employed that $v_p(b_4)<0$ and $v_p(b_3b_2+1)=0$.\\
\textbf{Induction step:}\\
Let us suppose that:
\[v_p(B_{3n-2})=v_p(B_{3n-1})=v_p(B_{3n})>v_p(B_{3n+1}).\]
In fact, the valuation of $B_{3n+1}$ is:
\[v_p(B_{3n+1})=v_p(b_{3n+1}B_{3n}+B_{3n-1})=v_p(b_{3n+1})+v_p(B_{3n}) < v_p(B_{3n}),\]
since, by induction hypotesis, $v_p(B_{3n})=v_p(B_{3n-1})$ and $v_p(b_{3n+1})<0$.\\ Recalling that $v_p(b_{3n+4})<0$ and $v_p(b_{3n+3}b_{3n+2}+1)=0$, at the following steps we obtain:
\begin{align*}
v_p(B_{3n+2})&=v_p(b_{3n+2}B_{3n+1}+B_{3n})=v_p(b_{3n+2})+v_p(B_{3n+1})=\\
&=v_p(B_{3n+1})<v_p(B_{3n}),\\
v_p(B_{3n+3})&=v_p(b_{3n+3}B_{3n+2}+B_{3n+1})=\\
&=v_p((b_{3n+3}b_{3n+2}+1)B_{3n+1}+b_{3n+3}B_{3n})=\\
&=v_p((b_{3n+3}b_{3n+2}+1)B_{3n+1})=v_p(B_{3n+1})=\\&=v_p(B_{3n+2})<v_p(B_{3n}),\\
v_p(B_{3n+4})&=v_p(b_{3n+4}B_{3n+3}+B_{3n+2})=v_p(b_{3n+4})+v_p(B_{3n+3})<\\
&<v_p(B_{3n+3})=v_p(B_{3n+1})=v_p(B_{3n+2}).
\end{align*}
Hence, we have obtained that
\[v_p(B_{3n+4})<v_p(B_{3n+3})=v_p(B_{3n+2})=v_p(B_{3n+1})<v_p(B_{3n}),\]
and this proves the claim.
\end{proof}
\end{Theorem}

Theorem \ref{ConvBr3} easily leads to the following corollary, achieving the convergence of a $p$--adic continued fraction generating the partial quotients as in (\ref{Br3}).

\begin{Corollary}\label{CorConvBr3}
Let $b_0,b_1,\ldots$ as in Theorem \ref{ConvBr3}. Then the continued fraction $[b_0,b_1,\ldots]$ is convergent to a $p$--adic number.
\begin{proof}
We know from Remark \ref{rema2} that the continued fraction $[b_0,b_1,\ldots]$ converges to a $p$--adic number if and only if
\[\lim\limits_{n\rightarrow +\infty} v_p(B_nB_{n+1})= -\infty.\]
Notice that, for all $n\in\mathbb{N}$,
\[v_p(B_{3n}B_{3n+1})>v_p(B_{3n+1}B_{3n+2}), \]
since $v_p(B_{3n+1})<v_p(B_{3n})$ and $v_p(B_{3n+1})=v_p(B_{3n+2})$. Then
\[v_p(B_{3n+1}B_{3n+2})=v_p(B_{3n+2}B_{3n+3}), \]
since all the three valuations are equal. Moreover,
\[v_p(B_{3n+2}B_{3n+3})>v_p(B_{3n+3}B_{3n+4}), \]
since $v_p(B_{3n+4})<v_p(B_{3n+3})$ and $v_p(B_{3n+3})=v_p(B_{3n+2})$. So, the sequence $v_p(B_nB_{n+1})$ is decreasing and divergent. 
\end{proof}
\end{Corollary}

\section{Some new algorithms}
Starting from Theorem \ref{ConvBr3} and Corollary \ref{CorConvBr3}, we propose some new algorithms. We use three different functions. For
\[a=\sum\limits_{n=-r}^{+\infty}a_np^n\in\mathbb{Q}_p, \ \ \ \ a_n\in\Big\{ 0, \pm 1,\pm 2,\ldots,\pm \frac{p-1}{2}\Big\},\]
the first two functions are the same $s$ and $t$ of \textit{Browkin II}, that are
\[
s(a)=\sum\limits_{n=-r}^{0}a_np^n, \ \ \ t(a)=\sum\limits_{n=-r}^{-1}a_np^n,\]
and then the third is:
\begin{align*}
u(a)=\begin{cases} +1 \ \  &\textup{if}  \ a_0\in\Big\{+2,\ldots,\dfrac{p-1}{2}\Big\}\cup \{-1\}\\
-1 &\textup{if}  \ a_0\in\Big\{-\dfrac{p-1}{2},\ldots,-2\Big\}\cup \{+1\}.
\end{cases}
\end{align*}

We can now design the shape of two new algorithms.

\begin{Definition}[First new algorithm]\label{firstnew}
On input $\alpha_0=\alpha$, for $n\geq 0$, our first new algorithm work as follows:
\begin{align*}\begin{cases}
b_n=s(\alpha_n) \ \ \ \ \ &\textup{if} \ n \equiv 0\bmod 3\\
b_n=t(\alpha_n) &\textup{if}  \ n \equiv 1 \bmod 3 \ \textup{and} \ v_p(\alpha_n-t(\alpha_n))= 0\\
b_n=t(\alpha_n)-sign(t(\alpha_n)) & \textup{if} \ n \equiv 1 \bmod 3 \ \textup{and} \ v_p(\alpha_n-t(\alpha_n))\neq0\\
b_n=u(\alpha_n)  & \textup{if} \ n \equiv 2 \bmod 3\\
\alpha_{n+1}=\frac{1}{\alpha_n-b_n}.
\end{cases}
\end{align*}
\end{Definition}

\begin{Definition}[Second new algorithm]\label{secondnew}
On input $\alpha_0=\alpha$, for $n\geq 0$, our second new algorithm work as follows:
\begin{align*}\begin{cases}
b_n=s(\alpha_n) \ \ \ \ \ &\textup{if} \ n \equiv 0\bmod 3\\
b_n=t(\alpha_n) &\textup{if}  \ n \equiv 1 \bmod 3 \ \textup{and} \ v_p(\alpha_n-t(\alpha_n))= 0\\
b_n=t(\alpha_n)-sign(t(\alpha_n)) & \textup{if} \ n \equiv 1 \bmod 3 \ \textup{and} \  v_p(\alpha_n-t(\alpha_n))\neq0\\
b_n=s(\alpha_n)-u(\alpha_n)  & \textup{if} \ n \equiv 2 \bmod 3\\
\alpha_{n+1}=\frac{1}{\alpha_n-b_n}.
\end{cases}
\end{align*}
\end{Definition}

\begin{Remark}
The choice of the third function  $u$ is a little tricky. The function $t$ takes all the negative powers, leaving out the constant term. The function $u$ needs to act on a $p$--adic number with zero valuation, but it has to leave apart another term with zero valuation, otherwise the third partial quotient will not have null valuation.

Clearly, the choice of this function can be done in several ways. In fact, there are a lot of manners to separate the constant term $a_0\in\{-\frac{p-1}{2},\ldots,\frac{p-1}{2} \}$ in two nonzero parts.
Here we have presented two proposals, but it would surely be interesting to analyze also other options different from ours.
\end{Remark}

Both of the constructions in Definition \ref{firstnew} and Definition \ref{secondnew} produce a sequence of partial quotients $b_0,b_1,\ldots\in\mathbb{Q}_p$ such that, for all $n\in\mathbb{N}$,
\[\begin{cases}
v_p(b_{3n+1})<0\\
v_p(b_{3n+2})=0\\
v_p(b_{3n})=0.
\end{cases}\]

We are going to see that also the additional condition required by Theorem \ref{ConvBr3}, i.e.
\[v_p(b_{3n+2}b_{3n+3}+1)=0, \ \textup{for} \ \textup{all} \ n \in \mathbb{N},\]
is satisfied for both algorithm.

\begin{Proposition}\label{Alg1}
Let $\alpha\in\mathbb{Q}_p$. Then the partial quotients generated by the new algorithms in Definition \ref{firstnew} and Definition \ref{secondnew} satisfy the conditions of Theorem \ref{ConvBr3}.
\begin{proof}
To prove the claim, we are left to show that     \[v_p(b_{3n+2}b_{3n+3}+1)=0, \, \text{ for all } n\in\mathbb{N}.\]
We prove it only for the second algorithm, the other proof is similar. First we notice that, by construction,
\[v_p(b_{3n+2}b_{3n+3})=v_p(b_{3n+2})+v(b_{3n+3})=0,\]
so that
$v_p(b_{3n+2}b_{3n}+1)\geq \min \{ v_p(b_{3n+2}b_{3n}), v_p(1)\}=0$.
Let us show that the case $v_p(b_{3n+2}b_{3n}+1)>0$ can not occur. 
For all $n\in\mathbb{N}$,
\[\alpha_{3n+2}= \frac{1}{\alpha_{3n+1}-t(\alpha_{3n+1})}=a_0+a_1p+a_2p^2+\ldots. \]
and 
\begin{align*}
b_{3n+2}&=s(\alpha_{3n+2})-u(\alpha_{3n+2})=a_0 \mp 1,\\
b_{3n+3}&=s(\alpha_{3n+3})=s\Big(\frac{1}{\alpha_{3n+2}-b_{3n+2}}\Big)=(a_0-b_{3n+2})^{-1}=\pm 1.
\end{align*}
Therefore, the condition $v_p(b_{3n+2}b_{3n}+1)=0$ is satisfied if and only if
\[b_{3n+2}(a_0-b_{3n+2})^{-1} \equiv (\pm 1) (a_0 \mp 1 ) \equiv -1 \bmod p\]
is not fulfilled.
However, this would imply that $a_0 \equiv 0 \bmod p$, but this cannot happen, due to the constraints in the algorithm when using the function $t$.
\end{proof}
\end{Proposition}

Finally we prove that the second new algorithm succeed in obtaining the finiteness of the expansion for rational numbers, as it happens for \textit{Browkin I} and \textit{Browkin II}. We state it in the following theorem.

\begin{Theorem}\label{finito}
If $\alpha \in \mathbb Q$, then the second new algorithm (Definition \ref{secondnew}) stops in a finite number of steps.
\end{Theorem}

\begin{proof}
Let us consider $\alpha\in\mathbb{Q}$. We are going to show that the algorithm from Definition \ref{secondnew} stops in a finite number of steps when the input is $\alpha$. By construction we have, 
\[v_p(\alpha_{3k+1})<0, \ v_p(\alpha_{3k+2})=v_p(\alpha_{3k+3})=0,\]
so that we can write
\begin{align*}
\alpha_{3k+1}&=\frac{N_{3k+1}}{D_{3k+1}p^l}, & \ &\text{with} \ (N_{3k+1},D_{3k+1})=1, \ \ p\not| N_{3k+1}D_{3k+1},  \ \ l\geq 1,\\
\alpha_{3k+2}&=\frac{N_{3k+2}}{D_{3k+2}}, & &\text{with} \ (N_{3k+2},D_{3k+2})=1, \ \ p\not| N_{3k+2}D_{3k+2},\\
\alpha_{3k+3}&=\frac{N_{3k+3}}{D_{3k+3}}, & &\text{with} \ (N_{3k+3},D_{3k+3})=1, \ \ p\not| N_{3k+3}D_{3k+3}.\\
\end{align*}
Let us notice that for this algorithm, for all $n\in\mathbb{N}$, the partial quotients are such that $b_{3n+2}\in\{-\frac{p-1}{2}+1,\ldots,-1,1,\ldots,\frac{p-1}{2}-1\}$ and $b_{3n+3}=\pm 1$, so that
\begin{equation*}
|b_{3n+2}|\leq \frac{p-3}{2},\quad
|b_{3n+3}|=1.
\end{equation*}
Since $v_p(b_{3n+1})<0$, we can write
\[b_{3n+1}=\frac{c_{3n+1}}{p^l}, \ \text{with} \ v_p(c_{3n+1})=0, \ l\geq 1.\]
The partial quotients $b_{3n+1}$ are generated by the function $t$ and it has been shown in \cite{BCMI} that
\[|c_{3n+1}|\leq p^l\left(1-\frac{1}{p^l}\right). \]
For the sake of simplicity, we also write $c_{3k+2}=b_{3k+2}$ and $c_{3k+3}=b_{3k+3}$, so that the coefficients $c_n$ always have zero valuation.\\
Exploiting $\alpha_{k+1}=\frac{1}{\alpha_k-b_k}$, we get
\begin{align*}
N_{3k+1}(N_{3k}-c_{3k}D_{3k})&=p^lD_{3k}D_{3k+1},\\
    N_{3k+2}(N_{3k+1}-c_{3k+1}D_{3k+1})&=p^lD_{3k+1}D_{3k+2},\\
N_{3k+3}(N_{3k+2}-c_{3k+2}D_{3k+2})&=D_{3k+2}D_{3k+3}.
\end{align*}
Since $(|N_n|,p|D_n|)=1$ for all $n\in\mathbb{N}$, then 
\[|N_{3k+1}|=|D_{3k}|, \ |N_{3k+2}|=|D_{3k+1}|,\ |N_{3k+3}|=|D_{3k+2}|,\]
and
\begin{align*}
|D_{3k+1}|&=\frac{|N_{3k}-c_{3k}D_{3k}|}{p^l}\leq \frac{|N_{3k}|+|c_{3k}D_{3k}|}{p^l} =\frac{1}{p^l}|N_{3k}|+\frac{1}{p^l}|D_{3k}|,\\
|D_{3k+2}|&=\frac{|N_{3k+1}-c_{3k+1}D_{3k+1}|}{p^l}\leq \frac{1}{p^l}|N_{3k+1}|+\left( 1- \frac{1}{p^l}\right)|D_{3k+1}|,\\
|D_{3k+3}|&=|N_{3k+2}-c_{3k+2}D_{3k+2}|\leq |N_{3k+2}|+\left(\frac{p-3}{2} \right)|D_{3k+2}|.
\end{align*}
By using the formulas above we may write
 
\begin{align*}
&|N_{3k+3}|+|D_{3k+3}|\leq |D_{3k+1}|+\frac{p-1}{2}|D_{3k+2}|\leq \\
&\leq  |D_{3k+1}|+\frac{p-1}{2}\left(\frac{1}{p^l}|N_{3k+1}|+ \frac{p^l-1}{p^l}|D_{3k+1}| \right)=\\
&=\frac{p-1}{2p^l}|N_{3k+1}|+\frac{p^{l+1}+p^l-p+1}{2p^l}|D_{3k+1}|\leq\\
&\leq \frac{p-1}{2p^l}|D_{3k}|+\frac{p^{l+1}+p^l-p+1}{2p^l} \cdot \left(\frac{1}{p^l}|N_{3k}|+\frac{1}{p^l}|D_{3k}|\right)=\\
&=\left(\frac{p^{l+1}+p^l-p+1}{2p^{2l}}\right)|N_{3k}|+\left(\frac{2p^{l+1}-p+1}{2p^{2l}}\right)|D_{3k}|.
\end{align*}

We have that $2p^{l+1}-p+1<2p^{2l}$, since $p^{2l} \geq p^l$ for every $l \geq 1$ and consequently we also have $p^{l+1}+p^l-p+1<2p^{2l}$.
Thus, we obtain, for all $k\in\mathbb{N}$, that
\[|N_{3k+3}|+|D_{3k+3}|<|N_{3k}|+|D_{3k}|.\]
Since the sequence $\{|N_{3n}|+|D_{3n}|\}_{n\in\mathbb{N}}$ is a strictly decreasing sequence of natural numbers it must be finite and hence $\alpha$ has a finite continued fraction.
\end{proof}
\section{Generalization to $n$ steps}
The aim of this section is to generalize Theorem \ref{ConvBr3} to a generic n-step algorithm. On this purpose, we also need several additional conditions on the valuations, thus we introduce the following notation for a family of sequences.
Let $n,m\in\mathbb{N}$, with $m\geq 2$, we define the family of sequences $U_m^{(n)}$  as
\[U_m^{(0)}=1, \ \ U_m^{(1)}=b_m, \ \  U_m^{(n+1)}=b_{m+n}U_m^{(n)}+U_m^{(n-1)}.  \]

\begin{Lemma}\label{seqden}
For every $n\geq 2$, the partial denominators $B_n$ can be obtained as:
\[B_n=U_2^{(n-1)}B_1+U_3^{(n-2)}B_0.\]
\begin{proof}
Let us prove the claim by induction on $n$. For $n=2$ and $n=3$ it holds since:
\begin{align*}
B_2&=b_2B_1+B_0=U_2^{(1)}B_1+U_3^{(0)}B_0,\\
B_3&=b_3B_2+B_1=(b_3b_2+1)B_1+b_3B_0=U_2^{(2)}B_1+U_3^{(1)}B_0.
\end{align*}
Now let us suppose that the claim holds at the steps $n$ and $n+1$, that is:
\begin{align*}
B_n&=U_2^{(n-1)}B_1+U_3^{(n-2)}B_0,\\
B_{n+1}&=U_2^{(n)}B_1+U_3^{(n-1)}B_0.
\end{align*}
We are going to show that it is true also for $B_{n+2}$. In fact:
\begin{align*}
B_{n+2}&=b_{n+2}B_{n+1}+B_{n}=\\
&=b_{n+2}(U_2^{(n)}B_1+U_3^{(n-1)}B_0)+(U_2^{(n-1)}B_1+U_3^{(n-2)}B_0)=\\
&=(b_{n+2}U_2^{(n)}+U_2^{(n-1)})B_1+(b_{n+2}U_3^{(n-1)}+U_3^{(n-2)})B_0=\\
&=U_2^{(n+1)}B_1+U_3^{(n)}B_0.
\end{align*}
It follows that the thesis is true for all $n\geq 2$.
\end{proof}
\end{Lemma}

\begin{Remark}\label{seqdenk}
Notice that Lemma \ref{seqden} holds also starting from a generic step $k$. It means that for all $k\in\mathbb{N}$ and $n\geq 2$,
\[B_{k+n}=U_{k+2}^{(n-1)}B_{k+1}+U_{k+3}^{(n-2)}B_k,\]
and the proof is similar to the case $k=0$ seen in Lemma \ref{seqden}.
\end{Remark}

\begin{Theorem}\label{ConvBrN}
Let us consider $r\in\mathbb{N^+}$ and $b_0,b_1,\ldots \in \mathbb{Q}_p$ such that, for all $n\in\mathbb{N}$:
\[\begin{cases}
v_p(b_{rn+1})<0\\
v_p(b_{rn+i})=0, \ \forall i\in\{2,\ldots,r\}.\\
\end{cases}.\]
Moreover let us suppose that, for all $n\in\mathbb{N}$,
\begin{align*}
v_p(U_{rn+2}^{(i)})&=0 \ \textup{for} \ \textup{all} \  i\in \{2,\ldots,r-1\} \text{ and for } r \geq 3,\\
v_p(U_{rn+3}^{(i)})&=0 \ \textup{for} \ \textup{all}\  i\in \{2,\ldots,r-2\} \text { and for } r \geq 4.
\end{align*}

Then we have, for all $n\in\mathbb{N}$,
\[v_p(B_{rn+1})=v_p(B_{rn+2})=\ldots=v_p(B_{rn+r})>v_p(B_{rn+r+1}).\]
\begin{proof}
Let us prove the claim by induction on $n$.\\
\textbf{Base step:}\\
We prove the thesis for $n=0$. The valuation of the first denominator is:
\[v_p(B_1)=v_p(b_1)<0.\]
By Lemma \ref{seqden}, for $i\in\{2,\ldots,r\}$,
\begin{align*}
v_p(B_i)&=v_p(U_2^{(i-1)}B_1+U_3^{(i-2)}B_0)=v_p(U_2^{(i-1)}B_1)=v_p(b_2B_1)=v_p(B_1).
\end{align*}
At the following step, since $v_p(b_{r+1})<0$, we get:
\[v_p(B_{r+1})=v_p(b_{r+1}B_r+B_{r-1})=v_p(b_{r+1})+v_p(B_r)<v_p(B_r).\]
Hence, the claim is true for $n=0$.\\
\textbf{Induction step:}\\
Let us suppose that the thesis holds for a generic $n\in\mathbb{N}$, that is:
\[v_p(B_{rn+1})=v_p(B_{rn+2})=\ldots=v_p(B_{rn+r})>v_p(B_{rn+r+1}).\]
We want to prove the claim for $n+1$.
Here we use Remark \ref{seqdenk} with $k=r(n+1)$.
Now, for $i\in\{2,\ldots,r\}$,
\begin{align*}
v_p(B_{r(n+1)+i})&=v_p(U_{r(n+1)+2}^{(i-1)}B_{r(n+1)+1}+U_{r(n+1)+3}^{(i-2)}B_{r(n+1)})=\\
&=v_p(U_{r(n+1)+2}^{(i-1)}B_{r(n+1)+1})=\\
&=v_p(U_{r(n+1)+2}^{(i-1)})+v_p(B_{r(n+1)+1})=v_p(B_{r(n+1)+1}).
\end{align*}
At the following step, since $v_p(b_{r(n+2)+1})<0$, then:
\begin{align*}
v_p(B_{r(n+2)+1})&=v_p(b_{r(n+2)+1}B_{r(n+2)}+B_{r(n+2)-1})=\\
&=v_p(b_{r(n+2)+1}B_{r(n+2)})<v_p(B_{r(n+2)}).
\end{align*}
The induction is then complete and the claim holds for all $n\in\mathbb{N}$.
\end{proof}
\end{Theorem}

\begin{Corollary}\label{nsteps}
Let $r\in\mathbb{N^+}$ and $b_0,b_1,\ldots$ as in Theorem \ref{ConvBrN}. Then the continued fraction $[b_0,b_1,\ldots]$ is convergent to a $p$--adic number.
\begin{proof}
Using Remark \ref{rema2}, the continued fraction $[b_0,b_1,\ldots]$ converges in $\mathbb{Q}_p$ if and only if
\[\lim\limits_{n\rightarrow +\infty} v_p(B_nB_{n+1})= -\infty.\]
By Theorem \ref{ConvBrN} we have that, for all $n\in\mathbb{N}$,
\[v_p(B_{rn+1}B_{rn+2})=\ldots=v_p(B_{rn+r-1}B_{rn+r})>v_p(B_{rn+r}B_{rn+r+1}),\]
so that the sequence $v_p(B_nB_{n+1})$ is decreasing and divergent to $-\infty$.
\end{proof}
\end{Corollary}

By Corollary \ref{nsteps}, we obtain the convergence of a $p$--adic continued fractions algorithm generating the partial quotients as
\begin{equation}
\begin{cases}
v_p(b_{rn+1})<0\\
v_p(b_{rn+2})=0\\
v_p(b_{rn+3})=0\\
\ldots\\
v_p(b_{rn+r})=0.
\end{cases}
\end{equation}
With a construction similar to the one made in Example \ref{controex}, it can be proved that the conditions of Theorem \ref{ConvBrN} are necessary for the $p$--adic convergence.

\section{Conclusions}
In this paper we have analyzed the convergence of $p$--adic continued fractions in order to give a better understanding for the design of an optimal algorithm, that at the present time does not exist. In Theorem \ref{teoconve}, we have characterized the strict decrease of the valuations $v_p(B_n B_{n+1})$, used by Browkin in \cite{BI} and \cite{BII}. This characterization guarantees the $p$--adic convergence of all the algorithms generating partial quotients such that $v_p(b_n)+v_p(b_{n+1})<0$ for all $n\in\mathbb{N}$. Outside from this hypothesis, we have also obtained some effective conditions  for the convergence of a $p$--adic continued fractions with only one negative partial quotient each $r$ steps. In particular, Browkin's continued fractions in \cite{BI} and \cite{BII} are respectively the cases when $r=1$ and $r=2$. For the case $r=3$ we have proposed some actual algorithms, proving that one of them terminates in a finite number of steps when processing a rational number.

\end{document}